\documentclass[12pt]{article}
\usepackage{amssymb, amsthm}

\newtheorem{thm}{Theorem}

\newtheorem{theorem}{Theorem}[section]
\newtheorem{corollary}[theorem]{Corollary}

\newtheorem{lemma}[theorem]{Lemma}

\newtheorem{construction}[theorem]{Construction}

\def\card#1{\vert #1 \vert}
\def\gpindex#1#2{\card {#1\colon #2}}
\def\irr#1{{\rm  Irr}(#1)}

\def\ker#1{{\rm ker} (#1)}

\def\NN{{\cal N}}
\def\MM{{\cal M}}
\def\MMM{{\cal M}}
\def\TT{{\cal T}}
\def\UU{{\cal U}}

\def\BB{{\cal B}}
\def\phi{\varphi}

\newcommand{\stab}[2]{{#1}_{#2}}

\newcommand{\Bpi}[1]{{\rm B}_\pi (#1)}
\newcommand{\Ipi}[1]{{\rm I}_\pi (#1)}
\newcommand{\Dpi}[1]{{\rm D}_\pi (#1)}
\newcommand{\ZZ}{{\cal Z}}
\newcommand{\SSS}{{\cal S}}
\newcommand{\AAA}{{\cal A}}
\newcommand{\pair}[1]{{\rm pair} (#1)}
\newcommand{\LL}{{\cal L}}

\begin{document}

\title{Lifts of partial characters with respect to a chain of normal subgroups}

\author {
       Mark L.\ Lewis
    \\ {\it Department of Mathematical Sciences, Kent State University}
    \\ {\it Kent, Ohio 44242}
    \\ E-mail: lewis@math.kent.edu
       }
\date{December 11, 2008}

\maketitle

\begin{abstract}
In this paper, we consider lifts of $\pi$-partial characters with
the property that the irreducible constituents of their restrictions
to certain normal subgroups are also lifts.  We will show that such
a lift must be induced from what we call an inductive pair, and
every character induced from an inductive pair is a such a lift.
With this condition, we will get a lower bound on the number of such
lifts.

MSC primary : 20C15, secondary : 20C20

Keywords : $\pi$-partial characters, Brauer characters, lifts
\end{abstract}

\section{Introduction}

In his seminal paper \cite{pisep}, Isaacs introduced the idea of
$\pi$-partial characters that for $\pi$-separable groups are an
analog of $p$-Brauer characters for $p$-solvable groups.  Another
approach to $\pi$-partial characters, which is the one we will take,
can be found in \cite{pipart}.

Let $\pi$ be a set of primes and let $G$ be a $\pi$-separable group.
Take $G^0$ to be the set of $\pi$-elements in $G$.  The
$\pi$-partial characters of $G$ are to defined to be the
restrictions to $G^0$ of all of the characters of $G$.  A
$\pi$-partial character is called irreducible if it cannot be
written as the sum of other $\pi$-partial characters, and we write
$\Ipi G$ for the set of irreducible $\pi$-partial characters of $G$.

If $\chi$ is a character of $G$, then we write $\chi^0$ for the
restriction of $\chi$ to $G^0$.  When $\phi$ is a $\pi$-partial
character of $G$, we say that $\chi$ is a {\it lift} of $\phi$ if
$\chi^0 = \phi$.  In general, we say $\chi$ is a {\it $\pi$-lift} if
$\chi^0 \in \Ipi G$.  Because of the way we constructed the
$\pi$-characters, we know that all $\pi$-partial characters have
$\pi$-lifts. One main thread of research in this subject has been to
find sets of canonical lifts for $\Ipi G$. A set of $\pi$-lifts
${\cal A}$ is called canonical if restriction is a bijection from
${\cal A}$ to $\Ipi G$, and there is some canonical property used to
determine ${\cal A}$. The prototypical example of a canonical set of
$\pi$-lifts is the set $\Bpi G$ that was defined in \cite{pisep} by
Isaacs. If $2$ is not in $\pi$, then using a suggestion of Dade,
Isaacs has constructed in \cite{indres} a second set of canonical
$\pi$-lifts $\Dpi G$.  In our papers \cite{chains} and
\cite{lattice}, we discuss two other ways to construct sets of
canonical $\pi$-lifts.

Recently, J. P. Cossey has looked at a different question regarding
lifts.  In particular, he has focused on a single irreducible
$\pi$-partial character $\phi$, and he has studied the set of all
lifts of $\phi$ in \cite{bound}, \cite{vertex}, and \cite{behav}. In
so doing, he has asked a number of questions regarding the lifts of
$\phi$. One question he has asked in a talk at the 2007 Zassenhaus
group theory conference (see \cite{zass}) is whether we can
characterize all the lifts of $\phi$, and in particular, whether
there is some character pair from which all the lifts of $\phi$ are
induced. One purpose of this note is to present an example that
shows that no such pair needs to exist.

However, we actually wish to do more.  We wish to study further the
question of characterizing all of the lifts of a $\pi$-partial
character.  At this time, the question of characterizing all of the
lifts of $\phi$ seems out of reach. Thus, it seems reasonable as a
first step to add additional conditions that the lifts must satisfy
and then characterize the lifts that satisfy these additional
conditions. For example, one condition that it makes sense to look
at are lifts of $\phi$ whose restrictions to every normal subgroup
have irreducible constituents that are $\pi$-lifts. The $\pi$-lifts
in $\Bpi G$ and $\Dpi G$ have this property.  It seems reasonable to
ask if there are any other $\pi$-lifts with this property.

Rather than look at all normal subgroups at once, we restrict our
attention to a single normal series. A {\it normal $\pi$-series} for
$G$ is a set $\NN = \{ 1 = N_0, N_1, \dots, N_n = G \}$ of normal
subgroups in $G$ where $N_i \le N_{i+1}$ and $N_{i+1}/N_i$ is either
a $\pi$-group or a $\pi'$-group for each $i$. We say $\chi$ is an
{\it $\NN$-$\pi$-lift} if for each $N \in \NN$, the irreducible
constituents of $\chi_N$ are $\pi$-lifts. In particular, since $G
\in \NN$, this requires that $\chi$ be a $\pi$-lift.

In this note, we characterize the $\NN$-$\pi$-lifts of $\phi$.
Notice that $\chi$ is an $\NN$-lift for all $\pi$-series $\NN$ if
and only if the irreducible constituents of $\chi_N$ are $\pi$-lifts
for every normal subgroup $N$ of $G$, and so, by characterizing the
$\NN$-$\pi$-lifts for all $\NN$, we obtain a characterization of the
lifts whose restrictions to normal subgroups have irreducible
constituents that are $\pi$-lifts.  In \cite{chains} and
\cite{lattice}, we discussed two ways of constructing canonical
$\pi$-lifts based on $\NN$.  We will use $\Bpi {G : \NN}$ to denote
the canonical $\pi$-lifts developed in \cite{chains} and $\Bpi {G :
*\NN}$ for the $\pi$-lifts found in \cite{lattice}.

To do this, we will use some tools developed in \cite{chains}. Let
$G$ be a $\pi$-separable group and $\NN$ a $\pi$-series for $G$. For
every character $\chi \in \irr G$, we associated in \cite{chains} a
character pair $(T,\tau)$ which we called a {\it self-stabilizing
pair} with respect to $\NN$.  (A character pair $(T,\tau)$ is a
subgroup $T \le G$ and character $\tau \in \irr T$.  The
self-stabilizing pair is well-defined only up to conjugacy.) This
pair has the properties of a nucleus in that $\tau^G = \chi$ and
$\tau$ factors into a product of a $\pi$-special character and a
$\pi'$-special character. In this paper, we will define what it
means for a pair $(V,\gamma)$ to be {\it inductive} for $\NN$.  With
this definition, we determine which characters in $\irr G$ are
$\NN$-$\pi$-lifts.

\begin{thm}\label{main1}
Let $\pi$ be a set of primes, let $G$ be a $\pi$-separable group,
and $\NN$ a normal $\pi$-series for $G$.  Consider a character $\chi
\in \irr G$ with self-stabilizing pair $(V,\gamma)$ with respect to
$\NN$.  Assume that $\gamma = \alpha \beta$ where $\alpha$ is
$\pi$-special and $\beta$ is $\pi'$-special. Then the following are
equivalent:
\begin{enumerate}
\item $\chi$ is an $\NN$-$\pi$-lift.
\item $(V,\gamma)$ is inductive with respect to $\NN$.
\item $(V,\alpha)$ is inductive with respect to $\NN$ and $\beta$ is
linear.
\end{enumerate}
\end{thm}

We can use this characterization of the $\NN$-$\pi$-lifts to obtain
a lower bound on the number of $\NN$-$\pi$-lifts.

\begin{thm}\label{main2}
Let $\pi$ be a set of primes, let $G$ be a $\pi$-separable group,
and $\NN$ a normal $\pi$-series for $G$.  Fix $\pi$-partial
character $\phi \in \Ipi G$ and character $\chi \in \Bpi {G : \NN}$
so that $\chi^0 = \phi$.  Let $(V,\gamma)$ be a self-stabilizing
pair for $\chi$.  Then the following are true.

\begin{enumerate}
\item There is an injective map from the linear characters of $V/V'$
with $\pi'$-order to the $\NN$-$\pi$-lifts of $\phi$.
\item The number of $\NN$-$\pi$-lifts for $\phi$ is at least
$\gpindex V{V'}_{\pi'}$.
\end{enumerate}
\end{thm}

We return to our initial question of classifying the $\pi$-lifts
with the property that the restriction to every normal subgroup has
constituents that are lifts.  These $\pi$-lifts can now be
classified since they are precisely the $\pi$-lifts that are
$\NN$-$\pi$-lifts for all normal $\pi$-series $\NN$.  Unfortunately,
it is not clear whether this classification is practical.

Acknowledgement:  We would like to thank J. P. Cossey and I. M.
Isaacs for several useful discussions regarding this paper.

\section{Characterizing $\NN$-lifts}


For now, we will let $\NN$ be a collection of normal subgroups of
$G$.  We say $\LL$ is a {\it compatible} set of $\pi$-lifts for
$\NN$ if the following hold:
\begin{enumerate}
\item For each $N \in \NN$, we have $\LL (N) \subseteq \irr N$ and
$\alpha \mapsto \alpha^0$ is bijection from $\LL (N)$ to $\Ipi N$.
\item If $M \le N$ with $M,N \in \NN$, then the irreducible
constituents of $\alpha_M$ lie in $\LL (M)$ for all $\alpha \in \LL
(N)$.
\end{enumerate}

In \cite{pisep}, Isaacs proved that if $\chi \in \Bpi G$ and $N$ is
a normal subgroup of $G$, then the irreducible constituents of
$\chi_N$ lie in $\Bpi N$.  It follows that $\Bpi {\cdot}$ is a
compatible set of $\pi$-lifts for all $\NN$. Similarly, when $2
\not\in \pi$, Isaacs proved in \cite{indres} if $\chi \in \Dpi G$
and $N$ is any normal subgroup that the irreducible constituents of
$\chi_N$ lie in $\Dpi N$, and so, $\chi$ will be a $\NN$-$\pi$-lift
for all $\NN$. Hence, $\Dpi {\cdot}$ is a compatible set of
$\pi$-lifts for all $\NN$. The results proved in \cite{chains} and
\cite{indlift} imply that both $\Bpi {\cdot: \NN}$ and $\Bpi {\cdot
: * \NN}$ are both compatible sets of $\pi$-lifts for $\NN$ when
$\NN$ is a normal $\pi$-series. We now show that compatible sets of
$\pi$-lifts are compatible with conjugation.

\begin{lemma}\label{conjugacy}
Let $\pi$ be a set of primes, let $G$ be a $\pi$-separable group,
let $\NN$ be a collection of normal subgroups of $G$, and let $\LL$
be a compatible set of $\pi$-lifts for $\NN$.  If $N \in \NN$ and
$\alpha \in \LL (N)$, then $\alpha^g \in \LL (N)$ for all $g \in G$.
\end{lemma}

\begin{proof}
We know that $\alpha^0 \in \Ipi N$. Let $\phi \in \Ipi G$ so that
$\alpha^0$ is a constituent of $\phi_N$. Take $\chi \in \LL (G)$ so
that $\chi^0 = \phi$.  Now, $\alpha^0$ is an irreducible constituent
of $\phi_N = (\chi^0)_N = (\chi_N)^0$. Since every irreducible
constituent of $\chi_N$ lies in $\LL (N)$ and $\alpha$ is the unique
element of $\LL (N)$ that is a $\pi$-lift of $\alpha^0$, we deduce
that $\alpha$ is a constituent of $\chi_N$. Also, $\alpha^g$ is a
constituent of $\chi_N$, so $\alpha^g \in \LL (N)$ for all $g \in
G$.
\end{proof}

One of the motivations for looking at compatible sets of lifts is
that if we have a $\pi$-partial character of a normal subgroup in
the collection, then its compatible lift will have the same
stabilizer in G.

\begin{corollary}
Let $\pi$ be a set of prime, let $G$ be a $\pi$-separable group, let
$\NN$ be a collection of normal subgroups of $G$, and let $\LL$ be a
compatible set of $\pi$-lifts for $\NN$.  If $N \in \NN$ and $\alpha
\in \LL (N)$, then $\stab G{\alpha} = \stab G{\alpha^0}$.
\end{corollary}

\begin{proof}
It is obvious that if $g$ stabilizes $\alpha$, then it stabilizes
$\alpha^0$, so $\stab G{\alpha} \le \stab G{\alpha^0}$.  Suppose now
that $g$ stabilizes $\alpha^0$.  By Lemma \ref{conjugacy}, we have
that $\alpha^g \in \LL (N)$.  Now, $(\alpha^g)^0 = (\alpha^0)^g =
\alpha^0$.  Since restriction is a bijection from $\LL (N)$ to $\Ipi
N$, it follows that $\alpha^g = \alpha$.  We conclude that $\stab
G{\alpha} = \stab G{\alpha^0}$.
\end{proof}

We will now specialize to the case where $\NN$ is a normal
$\pi$-series.  At this time, we need to review the definition and
results regarding self-stabilizing pairs that appeared in
\cite{chains}. Let $\NN = \{ 1 = N_0 \le N_1 \le \cdots \le N_n = G
\}$ be a normal $\pi$-series for $G$. Fix $\chi \in \irr G$. A {\it
character tower} for $\chi$ with respect to $\NN$ is a set of
characters $\UU = \{ 1 = \nu_0, \nu_1, \dots, \nu_n = \chi \}$ where
each $\nu_i \in \irr {N_i}$ and $\nu_i$ is a constituent of
$(\nu_{i+1})_{N_i}$ for all $i = 0, \dots, n-1$.  We take $T$ to be
$G_\UU$, the stabilizer in $G$ of $\UU$.  We proved in \cite{chains}
that there is a unique character $\tau \in \irr T$ so that $\tau^G =
\chi$ and $\tau_{T \cap N_i} = t_i \tau_i$ for some character
$\tau_i \in \irr {T \cap N_i}$ with $(\tau_i)^{N_i} = \nu_i$.  We
say that $(T,\tau)$ is the {\it self-stabilizing pair} determined by
$\UU$.  It is proved in \cite{chains} that all the self-stabilizing
pairs determined by character towers for $\chi$ are conjugate in
$G$.  We also prove that $\tau$ must be $\pi$-factored. We defined
$\Bpi {G : \NN}$ to be the characters in $\irr G$ who have a
self-stabilizing pair $(T,\tau)$ where $\tau$ is $\pi$-special.

We continue to suppose that $\phi \in \Ipi G$.  We now show that if
$\psi \in \LL (G)$ with $\psi^0 = \phi$ and $\chi$ is any
$\NN$-$\pi$-lift of $\phi$, then the subgroup for a self-stabilizing
pair for $\chi$ is contained in the subgroup of a self-stabilizing
pair for $\psi$.

\begin{lemma}\label{commonchains}
Let $\pi$ be a set of primes, let $G$ be a $\pi$-separable group,
let $\NN$ be a normal $\pi$-series for $G$, and let $\LL$ be a
compatible set of $\pi$-lifts for $\NN$.  Assume $\phi$ is a
$\pi$-partial character of $G$, and fix $\psi \in \LL (G)$ so that
$\psi^0 = \phi$. Suppose $\chi$ is an $\NN$-$\pi$-lift of $\phi$.
Let $\TT = \{ 1 = \tau_0, \tau_1, \dots, \tau_n = \psi \}$ and $\UU
= \{ 1 = \nu_0, \nu_1, \dots, \nu_n = \chi \}$ be character towers
with respect to $\NN$ for $\psi$ and $\chi$ respectively so that
$\tau_i^0 = \nu_i^0$ for all $i$. Then $\stab G{\UU} \le \stab
G{\TT}$.
\end{lemma}

\begin{proof}
If $x \in G_{\UU}$, then $\nu_i^x = \nu_i$ for each $i$.  This
implies that $(\nu_i^x)^0 = \nu_i^0 = \tau_i^0$ for each $i$. On the
other hand, if $g \in G^0$, then $(\nu_i^x)^0 (g) = (\nu_i)
(g^{x^{-1}}) = (\tau_i) (g^{x^{-1}}) = (\tau_i^x)^0 (g)$, and this
implies that $\tau_i^0 = (\tau_i^x)^0$ for each $i$.  We now note
that $\tau_i, \tau_i^x \in \LL (N_i)$ using Lemma \ref{conjugacy}.
Since restriction to $G^0$ is a bijection from $\LL (N_i)$ to $\Ipi
{N_i}$, this implies that $\tau_i = \tau_i^x$ for all $i$. We
conclude that $x \in G_{\TT}$ as desired.
\end{proof}

We now show that if $\chi \in \LL (G)$ and $(T,\tau)$ is a
self-stabilizing pair for $\chi$ with respect to $\NN$, then there
is a $\pi$-special character $\sigma \in \irr T$ so that
$(T,\sigma)$ is a self-stabilizing pair for $\NN$ and $(\sigma^G)^0
= \chi^0$.

\begin{lemma} \label{degree}
Let $\pi$ be a set of primes, let $G$ be a $\pi$-separable group,
let $\NN$ be a normal $\pi$-series for $G$, and let $\LL$ be a
compatible set of $\pi$-lifts for $\NN$.  Suppose $\chi \in \LL (G)$
and $(T,\tau)$ is a self-stabilizing pair for $\chi$ with respect to
$\NN$.  Then $\tau (1)$ is a $\pi$-number.  Furthermore, there is a
character $\sigma \in \irr T$ so that $\sigma$ is $\pi$-special,
$(T,\sigma)$ is a self-stabilizing pair with respect to $\NN$, and
$(\sigma^G)^0 = \chi^0$.
\end{lemma}

\begin{proof}
Label $\NN = \{ 1 = N_0 \le N_1 \le \cdots \le N_n = G \}$, and let
$\TT = \{ 1 = \theta_0, \theta_1, \dots, \theta_n = \chi \}$ be the
character tower under $\chi$ that yields $(T,\tau)$.  We know that
$\theta_i \in \LL (N_i)$ for each $i$, and so, $\phi_i = \theta_i^0
\in \Ipi {N_i}$.  Recall that $T = \stab G{\TT}$.

We showed in \cite{chains} that restriction is a bijection from
$\Bpi {G : \NN}$ to $\Ipi G$.  Since $\chi^0 \in \Ipi G$, there is a
character $\psi \in \Bpi {G \mid \NN}$ so that $\psi^0 = \chi^0$.
Let $\NN_i = \{ N_0 \le N_1 \le \dots \le N_i \}$. Also, it was
proved in Theorem 5.2 of \cite{chains} that all the irreducible
constituents of $\psi_{N_i}$ lie in $\Bpi {N_i : \NN_i}$.  Let
$\gamma_i$ be the unique character in $\Bpi {N_i : \NN_i}$ such that
$\gamma_i^0 = \phi_i$. Since $\phi_i$ is a constituent of
$(\psi^0)_{N_i} = (\chi^0)_{N_i}$, it follows that $\gamma_i$ is a
constituent of $\psi_{N_i}$.  In fact, if $j \ge i$, then $\gamma_i$
will be a constituent of $(\gamma_j)_{N_i}$.  We conclude that $\UU
= \{ \gamma_0, \gamma_1, \dots, \gamma_n \}$ is a character tower
for $\NN$ under $\psi$.  Let $(S,\sigma)$ be the self-stabilizing
pair determined by this character tower, and recall that $S = \stab
G{\UU}$.  Since $\chi \in \LL (G)$, we have $S \le T$ by Lemma
\ref{commonchains}.  On the other hand, since $\Bpi {\cdot : \NN}$
is a compatible set of $\pi$-lifts for $\NN$, we obtain $T \le S$ by
Lemma \ref{commonchains}.  We conclude that $S = T$. We now have
$|G:T| \tau (1) = \chi (1) = \psi (1) = |G:T| \sigma (1)$, and $\tau
(1) = \sigma (1)$. Since $\psi \in \Bpi {G : \NN}$, we know that
$\sigma$ is $\pi$-special, and therefore, $\sigma (1)$ is a
$\pi$-number.
\end{proof}

For our characterization, we need one further definition.  The
character pair $(V,\gamma)$ is called {\it inductive for $\NN$} if
for each $N \in \NN$, we have $\gamma_{V \cap N} = a \eta$ for some
positive integer $a$ and character $\eta \in \irr {V \cap N}$ where
$(\eta^N)^0$ is irreducible.  This gives one implication of Theorem
\ref{main1}.

\begin{lemma}\label{inductive}
Let $\pi$ be a set of primes, let $G$ be a $\pi$-separable group,
and let $\NN$ be a normal $\pi$-series for $G$.  Let $(V,\gamma)$ be
inductive for $\NN$.  Then $\gamma^G \in \irr G$ is an
$\NN$-$\pi$-lift and $\gamma$ is $\pi$-factored.
\end{lemma}

\begin{proof}
Since $G \in \NN$, we have $(\gamma^G)^0$ is irreducible, and thus,
$\gamma^G \in \irr G$ will be a $\pi$-lift.  For each $N \in \NN$,
we have that $\gamma_{V \cap N} = a\eta$ for a positive integer $a$
and character $\eta \in \irr {V \cap N}$. We know that $(\eta^N)^0$
is irreducible. Notice that $\eta^N$ is an irreducible constituent
of $\chi_N$. Since the irreducible constituents of $\chi_N$ are all
conjugate, this implies that the irreducible constituents of
$\chi_N$ are $\pi$-lifts, and hence, $\chi$ is an $\NN$-$\pi$-lift.
The fact that $\gamma$ is $\pi$-factored follows from Theorem 21.7
of \cite{MaWo}.
\end{proof}

The next lemma is useful for identifying pairs that are inductive
for $\NN$.  This gives another implication of Theorem \ref{main1}.

\begin{lemma}\label{factored}
Let $\pi$ be a set of primes, let $G$ be a $\pi$-separable group,
and let $\NN$ be a normal $\pi$-series for $G$.  Let $(V,\gamma)$ be
a pair so that $\gamma$ factors as $\alpha\beta$ where $\alpha$ is
$\pi$-special and $\beta$ is $\pi'$-special.  If $\beta$ is linear,
then $(V,\alpha)$ is inductive for $\NN$ if and only if $(V,\gamma)$
is inductive for $\NN$.
\end{lemma}

\begin{proof}
Suppose first that $(V,\gamma)$ is inductive for $N$.  For each $N
\in \NN$, it follows that $\gamma_{V \cap N} = (\alpha\beta)_{V \cap
N}$ is homogeneous. We deduce that $\alpha_{V \cap N} = e \eta$ for
some $\pi$-special character $\eta$ of $V \cap N$ and positive
integer $e$.  Since $\beta$ is linear and $\pi'$-special, we have
$\beta_{V \cap N} = \xi$ where $\xi$ is linear and $\pi'$-special.
Hence, we have $(\alpha\beta)_{V \cap N} = e \eta \xi$. Because
$(V,\gamma)$ is inductive for $\NN$, we know that $((\eta\xi)^N)^0$
is irreducible. Also, $(\eta\xi)^0 = \eta^0$ as $\xi$ is linear and
$\pi'$-special, so $(\eta^N)^0 = ((\eta\xi)^N)^0$ is irreducible. We
conclude that $(V,\alpha)$ is inductive for $\NN$ as desired.

We now suppose that $(V,\alpha)$ is inductive for $\NN$.  Thus, for
$N \in \NN$, we have that $\alpha_{V \cap N} = a \eta$ for a
positive integer $a$ and character $\eta \in \irr {V \cap N}$ and
$(\eta^N)^0$ is irreducible. Since $\beta$ is linear and
$\pi'$-special, $\beta_N = \xi$ is irreducible and $\pi'$-special.
It follows that $\gamma_{V \cap N} = a \eta \xi$ and $(\eta \xi)^0 =
\eta^0$.  We deduce that $((\eta\xi)^N)^0 = (\eta^N)^0$ is
irreducible.  This implies that $(V,\gamma)$ is inductive.
\end{proof}

We have shown that if $\chi$ is an induced from a pair that is
inductive with respect to $\NN$, then $\chi$ is an $\NN$-$\pi$-lift.
We now show that if $\chi$ is an $\NN$-$\pi$-lift, then such a pair
must exist.  We now prove Theorem \ref{main1} which shows that
$\chi$ is an $\NN$-$\pi$-lift if and only if a self-stabilizing pair
$(V,\gamma)$ for $\chi$ with respect to $\NN$ is inductive for
$\NN$.  The more interesting portion of this theorem considers the
$\pi$-factorization of $\gamma$.  In this case, we show that the
$\pi'$-special factor of $\gamma$ must be linear, and the
$\pi$-special factor will afford an inductive pair for $\NN$.


\begin{proof} [Proof of Theorem \ref{main1}]
Suppose first that $\chi$ is an $\NN$-$\pi$-lift.  We show that
$(V,\alpha)$ is inductive.  Write $\NN = \{ 1= N_0, N_1, \dots,
N_n=G \}$, and let $\UU = \{ 1= \nu_0, \nu_1, \dots, \nu_n = \chi
\}$ be the character tower with respect to $\NN$ affording
$(V,\alpha\beta)$. We can find $\mu_i \in \Bpi {N_i}$ so that
$\mu_i^0 = \nu_i^0$.  Set $\psi = \mu_n$ and $\TT = \{ \mu_0, \mu_1,
\dots, \mu_n = \psi \}$. Notice that $\psi^0 = \chi^0$, and $\TT$ is
character tower for $\psi$ with respect to $\NN$.  Let $(T,\tau)$ be
the self-stabilizing pair afforded by $\TT$.  By Lemma
\ref{commonchains}, we see that $V \le T$, and by Lemma
\ref{degree}, we know that $\tau (1)$ is a $\pi$-number.  We have
$\chi (1) = |G:V| \alpha (1) \beta (1)$ and $\psi (1) = |G:T| \tau
(1)$.  Since $\chi (1) = \psi (1)$, this implies that $|G:V| \alpha
(1) \beta (1) = |G:T| \tau (1)$, and hence, $|T:V| \alpha (1)\beta
(1) = \tau (1)$.  Now, $\beta (1)$ divides the $\pi$-number $\tau
(1)$.  On the other hand, $\beta$ is $\pi'$-special, so $\beta (1)$
is a $\pi'$-number.  Together, these imply that $\beta (1) = 1$, and
so, $\beta$ is linear.

For each $N \in \NN$, since $(V,\alpha\beta)$ is a self-stabilizing
pair, it follows that $(\alpha\beta)_{V \cap N}$ is homogeneous.  We
deduce that $\alpha_{V \cap N} = e \eta$ for some $\pi$-special
character $\eta$ of $V \cap N$ and positive integer $e$.  Since
$\beta$ is linear, we have $\beta_{V \cap N} = \xi$ where $\xi$ is
linear.  Hence, we have $(\alpha\beta)_{V \cap N} = e \eta \xi$.
Because $(V,\alpha\beta)$ is self-stabilizing for $\chi$, we know
that $\nu = (\eta\xi)^N$ is an irreducible constituent of $\chi_N$.
Since $\chi$ is an $\NN$-$\pi$-lift, this implies that $\nu^0$ is
irreducible.  Also, $(\eta\xi)^0 = \eta^0$ as $\xi$ is linear and
$\pi'$-special, so $(\eta^N)^0 = ((\eta\xi)^N)^0 = \nu^0$ is
irreducible.  We conclude that $(V,\alpha)$ is inductive for $\NN$
as desired.

If $(V,\alpha)$ is inductive for $\NN$ and $\beta$ is linear, then
Lemma \ref{factored} implies that $(V,\gamma)$ is inductive for
$\NN$.  Finally, if $(V,\gamma)$ is inductive for $\NN$, then Lemma
\ref{inductive} implies that $\chi$ is an $\NN$-$\pi$-lift.
\end{proof}

\section{Counting $\NN$-lifts}

We now turn our attention to the lower bound on the number of
$\NN$-$\pi$-lifts which is proved in Theorem \ref{main2}.  To do
this, we need to gather more results regarding inductive pairs.  We
now show that if we have a pair that is inductive, then the subgroup
of that pair must be contained in the subgroup for some
self-stabilizing pair.

\begin{lemma}\label{containment}
Let $\pi$ be a set of primes, let $G$ be a $\pi$-separable group,
and let $\NN$ be a normal $\pi$-series for $G$.  Let $(V,\gamma)$ be
inductive for $\NN$, and write $\chi = \gamma^G \in \irr G$.  Then
$\chi$ has a self-stabilizing pair $(U,\delta)$ with respect to
$\NN$ so that $V \le U$.
\end{lemma}

\begin{proof}
Write $\NN = \{ 1 = N_0 \le N_1 \le \dots \le N_n = G \}$.  For each
$i$, we have $\gamma_{N_i} = a_1 \eta_i$ for some positive integer
$a_i$ and character $\eta_i \in \irr {V \cap N_i}$. Also, we have
$\nu_i = (\eta_i)^{N_i}$ satisfies $(\nu_i)^0$ is irreducible and
so, $\nu_i \in \irr {N_i}$. It is not difficult to see that $\UU =
\{ \nu_0, \nu_1, \dots, \nu_n \}$ is character tower for $\chi$ with
respect to $\NN$.  Let $(U,\delta)$ be the self-stabilizing pair
afforded by $\UU$.  Since $V$ stabilizes each $\eta_i$, we see that
$V$ stabilizes each $\nu_i$, and so, $V$ stabilizes $\UU$.  This
implies that $V \le U$.
\end{proof}

This yields the following useful observation.

\begin{corollary}
Let $\pi$ be a set of primes, let $G$ be a $\pi$-separable group,
and let $\NN$ be a normal $\pi$-series for $G$.  Let $(V,\gamma)$ be
a pair so that $\gamma$ factors as $\alpha\beta$ where $\alpha$ is
$\pi$-special and $\beta$ is $\pi'$-special.  Then $(V,\gamma)$ is
inductive for $\NN$ if and only if $\beta$ is linear and
$(V,\alpha)$ is inductive for $\NN$.
\end{corollary}

\begin{proof}
Suppose first that $\beta$ is linear and $(V,\alpha)$ is inductive
for $\NN$. Then $(V,\gamma)$ is inductive for $\NN$ by Lemma
\ref{factored}.

Conversely, we suppose that $(V,\gamma)$ is inductive.  By Lemma
\ref{inductive}, we know that $\chi = \gamma^G \in \irr G$ is an
$\NN$-$\pi$-lift. {}From Lemma \ref{containment}, $\chi$ has a
self-stabilizing pair $(U,\delta)$ so that $V \le U$. By Lemma
\ref{degree}, we know that $\delta (1)$ is a $\pi$-number.  We have
$|G:U| \delta (1) = \chi (1) = |G:V| \alpha (1) \beta (1)$, and
hence, $\delta (1) = |U:V| \alpha (1) \beta (1)$.  Thus, $\beta (1)$
divides $\delta (1)$. Since $\delta (1)$ is a $\pi$-number and
$\beta$ is $\pi'$-special, we conclude that $\beta (1) = 1$.  We now
apply Lemma \ref{factored} to see that $(V,\alpha)$ is inductive for
$\NN$.
\end{proof}

If $(V,\alpha)$ is a self-stabilizing pair for $G$ with respect to
$\NN$ where $\alpha$ is $\pi$-special, then we show that
$(V,\alpha)$ is inductive for $\NN$.

\begin{corollary} \label{selfind}
Let $\pi$ be a set of primes, let $G$ be a $\pi$-separable group,
and let $\NN$ be a normal $\pi$-series for $G$.  Suppose that
$(V,\alpha)$ is a self-stabilizing pair with respect to $\NN$. If
$\alpha$ is $\pi$-special, then $(V,\alpha)$ is inductive for $\NN$.
\end{corollary}

\begin{proof}
Let $\chi \in \irr G$ be the character associated with $(V,\alpha)$.
We proved in Theorem B and Theorem 5.2 of \cite{chains} that $\chi$
is an $\NN$-$\pi$-lift.  Applying Theorem \ref{main1}, we conclude
that $(V,\alpha)$ is inductive for $\NN$.
\end{proof}

Finally, we have a condition which will be used to find other
self-stabilizers pair using a given self-stabilizing pair.

\begin{lemma} \label{indself}
Let $\pi$ be a set of primes, let $G$ be a $\pi$-separable group,
and let $\NN$ be a normal $\pi$-series for $G$.  Suppose that
$(V,\alpha)$ is a self-stabilizing pair with respect to $\NN$ where
$\alpha$ is $\pi$-special.  Assume $\beta \in \irr V$ is
$\pi'$-special. If $\beta_{V \cap N}$ is homogeneous for all $N \in
\NN$, then $(V,\alpha\beta)$ is a self-stabilizing pair with respect
to $\NN$.
\end{lemma}

\begin{proof}
We work by induction on $|G|$.  Let $N$ be the penultimate term of
$\NN$.  Let $\alpha_{V \cap N} = a \eta$ and $\beta_{V \cap N} = b
\xi$.  Write $\NN^*$ for the normal $\pi$-series by terminating
$\NN$ at $N$.  We know that $(V \cap N,\eta)$ is a self-stabilizing
pair with respect to $\NN^*$ and that its stabilizer in $G$ is $V$.
By induction, $(V \cap N,\eta\xi)$ is a self-stabilizing pair with
respect to $\NN^*$.  Obviously, $\eta\xi$ is invariant in $V$ and
since $V$ is the stabilizer of $(V \cap N,\eta)$, it follows that it
stabilizes $(V \cap N,\eta\xi)$.  We conclude that $(V,\alpha\beta)$
is a self-stabilizing pair for $\NN$.
\end{proof}

For the next proof, we need the theory of source maps that we
developed in \cite{lattice}.  Following \cite{dadeind} and
\cite{indsou}, we say that a character pair $(H,\phi)$ is an {\it
inductive source} if induction is an injection from $\irr {T \mid
\phi}$ to $\irr G$ where $T$ is the stabilizer of $(H,\phi)$ in $G$.
Let $G$ be a finite group, and let $\pair G$ be the set of all
character pairs in $G$. We define {\it a source map} on $G$ to be a
function $\MMM : \pair G \rightarrow \pair G$ that satisfies for all
pairs $(H,\theta) \in \pair G$: (1) $\MMM (H,\theta) \le
(H,\theta)$, (2) $\MMM (H,\theta)$ is an inductive source in $H$,
and (3) if $\MMM (H,\theta) < (H,\theta)$, then $I_H (\MMM
(H,\theta)) < H$. Given a source map $\MMM$ on $G$ and a pair
$(H,\theta) \in \pair G$, we set $\MMM (H,\theta) = (M,\mu)$.  Let
$T$ be the stabilizer in $H$ of $(M,\mu)$.  Since $(M,\mu)$ is
inductive source for $H$, we know that induction is a bijection from
$\irr {T \mid \mu}$ to $\irr {H \mid \mu}$.  Since $\theta \in \irr
{M  \mid \mu}$, there is a unique character $\tau \in \irr {T \mid
\mu}$ so that $\tau^G = \theta$.  Thus, given $\MMM$, we can define
a second map $\SSS_\MMM: \pair G \rightarrow \pair G$ by $\SSS_\MMM
(H,\theta) = (T,\tau)$ as defined above.

In \cite{chains}, we defined a source map $\ZZ$ from $\pair G$ to
$\pair G$ as follows. Consider $(H,\theta) \in \pair G$. If
$\theta_{N \cap H}$ is homogeneous for all $N \in \NN$, then $\ZZ
(H, \theta) = (H,\theta)$.  Otherwise, define $\ZZ (H,\theta)$ to be
the $\pi$-factored pair $(N \cap H,\gamma)$ that is minimal such
that $N \in \NN$ and $\gamma$ is not $H$-invariant. We showed in
Lemma 4.3 of \cite{chains} that $\ZZ$ is a $\pi$-source map, i.e.
that $\ZZ (H,\theta) \le \ZZ (\SSS_\ZZ (H,\theta))$ for all
$(H,\theta) \in \pair G$, and that the nucleus determined by $\ZZ$
is a self-stabilizing pair.

We will say that a source map $\MM$ is {\it $\pi$-closed} if $\MM$
is a source map and the following condition holds.  If $(H,\theta)
\in \pair G$ with $\MM (H,\theta) = (S,\sigma)$ where $\sigma$ is
$\pi$-special, then for every $\pi'$-special character $\delta \in
\irr S$, the pair $(S,\sigma\delta)$ is an inductive source for $H$.
Note that this is a modification of the definition of $\pi$-closed
found in \cite{lattice}.  The following theorem is Theorem 2.1 of
\cite{lattice}.  It is a generalization of a result of Navarro:
Theorem A in \cite{newnav}.  We claim that the proof of this theorem
under our modified definition of $\pi$-closed is identical to the
proof using the definition found in \cite{lattice}.

\begin{theorem}\label{lat}
Let $\pi$ be a set of primes, let G be a finite $\pi$-separable
group, and let M be a $\pi$-closed source map on G so that $\MM (H,
\theta) \le \MM (\SSS_\MM (H, \theta))$ for all $(H, \theta) \in
\pair G$. Suppose $\chi \in \Bpi {G :M}$ with $\MM$-nucleus $(W,
\gamma)$. Let $\AAA$ be the set of $\pi'$-special characters of $W$.
Then the map $\delta \mapsto (\gamma\delta)^G$ is an injection from
$\AAA$ to $\irr G$.
\end{theorem}

We now have the following corollary to this theorem which yields for
self-stabilizing pairs a result similar to Navarro's theorem.

\begin{corollary} \label{map}
Let $\pi$ be a set of primes, let $G$ be a $\pi$-separable group,
and let $\NN$ be a normal $\pi$-series for $G$.  Suppose that
$(V,\alpha)$ is a self-stabilizing pair with respect to $\NN$ where
$\alpha$ is $\pi$-special.  Then the map $\beta \mapsto
(\alpha\beta)^G$ is an injection from the $\pi'$-special characters
of $V$ to $\irr G$.
\end{corollary}

\begin{proof}
Let $\chi = \alpha^G$.  We know that $(V,\alpha)$ is the
$\ZZ$-nucleus of $\chi$.  The result will follow from Theorem
\ref{lat} if we can show that $\ZZ$ is $\pi$-closed.  Suppose
$(H,\theta) \in \pair G$ and $\ZZ (H,\theta) = (S,\sigma)$ where
$\sigma$ is $\pi$-special.  We know that $S = H \cap N$ for some $N
\in \NN$.  This implies that $S$ is normal in $H$.  Thus, $(S,\sigma
\delta)$ will be an inductive source for all $\pi'$-special $\delta
\in \irr S$.
\end{proof}

We can now prove Theorem \ref{main2}.

\begin{proof}[Proof of Theorem \ref{main2}]
We have $\phi \in \Ipi G$ and $\chi \in \Bpi {G : \NN}$ so that
$\chi^0 = \phi$ where $(V,\gamma)$ is a self-stabilizer pair for
$\chi$.  Since $\chi \in \Bpi {G : \NN}$, we know that $\gamma$ is
$\pi$-special.  By Corollary \ref{selfind}, $(V,\gamma)$ is
inductive. Then applying Lemma \ref{indself}, we deduce that
$(V,\gamma \beta)$ is inductive if $\beta \in \irr {V/V'}$ with
$\pi'$-order.  Next, we use Corollary \ref{map} to see that there is
an injection from the characters in $\irr {V/V'}$ with $\pi'$-order
to $\irr G$ by $\beta \mapsto (\gamma \beta)^G$.  Observe that
$((\gamma \beta)^G)^0 = (\gamma^0 \beta^0)^G = (\gamma^0)^G =
(\gamma^G)^0 = \chi^0 = \phi.$ Hence, we have an injection from the
characters in $\irr {V/V'}$ with $\pi'$-order to the $\pi$-lifts of
$\phi$, and this gives the desired conclusion.
\end{proof}

\section{An example}

\medskip
In this section, we construct a group $G$, and a partial character
$\phi$ of $G$.  We will study the $\pi$-lifts of $\phi$, and show
that they have a number of properties that illustrate the work we
have done in this paper.  In particular, we will show that the
$\pi$-lifts of this group are not determined by a single subgroup
answering the question of Cossey that was raised in the
introduction.  We will also show that that the number of $\pi$-lifts
does not divide $\card G$ which was another question that Cossey has
raised.   We will show that for different $\pi$-normal series $\NN$,
one gets different sets of $\NN$-$\pi$-lifts.   We will also show
that this group has an inductive pair that is not self-stabilizing.

\begin{construction} \rm
Let $E$ be an extra-special group of order $3^3$ (and exponent $3$).
Also, take $V$ to be an elementary abelian group of order $7^2$.  We
write $V = V_1 \times V_2$ where $V_1$ and $V_2$ both have order
$7$.  Let $M_1$ and $M_2$ be maximal subgroups of $E$ so that $M_1
\ne M_2$. We define an action of $E$ on $V$ so that $E$ acts on
$V_1$ with $M_1$ as its kernel and $E$ acts on $V_2$ with $M_2$ as
its kernel.  Let $G$ be the semi-direct product of $E$ acting on
$V$. Hence, $|G| = 3^3 \cdot 7^2$.  Now, $\irr {G/V} = \irr E$ has
$9$ linear characters and $2$ characters of degree $3$. Also, $\irr
V \setminus \{ 1 \}$ has $4$ orbits of size $3$ (consisting of the
characters with either $V_1$ or $V_2$ as their kernels) and $4$
orbits of size $9$ under the action (the characters that do not have
either $V_1$ or $V_2$ as their kernels).  Each of these characters
extend to their stabilizers in $G$ to yield $36$ characters of
degree $3$ and $12$ characters of degree $9$ in $\irr G$ that do not
have $E$ in their kernels.
\end{construction}

We take $\pi = \{ 3 \}$.  It is not difficult to see that $\Bpi G =
\irr {G/V}$.  Fix a character $\chi \in \Bpi G$ with $\chi (1) = 3$.
Let $\phi = \chi^0$.

\begin{lemma}
There are $13$ lifts of $\phi$ in $\irr G$.
\end{lemma}

\begin{proof}
We want to count the number of lifts of $\phi$.  Let $X/V$ be the
center of $G/V$.  We know that $\chi$ is fully ramified with respect
to $G/X$, and we take $\hat\lambda \in \irr {X/V}$ to be the unique
irreducible constituent of $\chi_X$. It follows that $\hat\lambda^0$
is the unique irreducible constituent of $\chi^0_X$. Let $Z$ be the
center of $G$ and observe that $Z$ is the center of $E$. Notice that
$X = Z \times V$, and so, restriction is a bijection from $\irr
{X/V}$ to $\irr Z$. We write $\lambda = \hat\lambda_Z$, and we have
$\hat\lambda = \lambda \times 1_V$.  If $\eta \in \irr X$ with
$\eta^0 = \hat\lambda^0$, then it is not difficult to see that $\eta
= \lambda \times \alpha$ for some character $\alpha \in \irr V$.

Suppose $\psi$ is a lift of $\phi$.  Let $\xi$ be an irreducible
constituent of $\psi_X$.  Now, $X$ is abelian, so $\xi$ is linear.
This implies that $\xi^0$ is irreducible.  We see that $\xi^0$ is a
constituent of $\psi^0_X = \chi^0_X$.  This implies that $\xi^0 =
\hat\lambda^0$, and hence, $\xi = \lambda \times \alpha$ for some
character $\alpha \in \irr V$.  If $\alpha = 1_V$, then we have $\xi
= \hat\lambda$, and hence, $\psi = \chi$.  Suppose $\alpha \ne 1_V$.
If $\alpha$ lies in an orbit of size $9$ under the action of $E$,
then $9$ divides $\psi (1)$ (since $\alpha$ is a constituent of
$\psi_V$).  Since $\psi (1) = \chi (1) = 3$, this cannot occur.
Hence, $\alpha$ must lie in an orbit of size $3$.  This implies that
either $V_1$ or $V_2$ is in the kernel of $\alpha$.  Let $\BB = \{
\beta \in \irr V \mid \ker {\beta} = V_i ~{\rm for}~ i = 1,2 \}$. We
have shown that if $\psi$ is a lift of $\phi$ other than $\chi$,
then the irreducible constituents of $\psi_X$ have the form $\lambda
\times \beta$ where $\beta \in \BB$.

We now show that if $\psi \in \irr {G \mid \lambda \times \beta}$
where $\beta \in \BB$, then $\psi$ is a lift of $\phi$.  Without
loss of generality, we may assume that $\ker {\beta} = V_2$.  Since
$V = V_1 \times V_2$ and $M_1$ centralizes $V_1$, we have that $M_1$
stabilizes $\beta$.  Also, $|G:M_1V| = 3$, so $M_1V$ is the
stabilizer of $\beta$ in $G$.  Since $V$ is a Hall subgroup, we see
that $\beta$ has a canonical extension $\hat\beta \in \irr {M_1V}$.
By Clifford's and Gallagher's theorems, we know there is a unique
character $\delta \in \irr {M_1V/V}$ so that $(\delta \hat\beta)^G =
\psi$.  We know that $\delta$ is an extension of $\hat\lambda$, and
since $G/V$ is extra-special, we know that $\delta^G = \chi$.
Observe that $\hat\beta^0 = 1$, so $\psi^0 = ((\delta
\hat\beta)^G)^0 = (\delta^G)^0 = \chi^0 = \phi$.

Notice that $\BB$ contains $12$ characters which form $4$ orbits of
size $3$ under the action of $G$.  Fixing $\beta \in \BB$, we see
that $\lambda \times \beta$ has $3$ extensions to its stabilizer.
Thus, each of the orbits of $\BB$ yields $3$ lifts of $\phi$.  Thus,
we obtain in total $13$ lifts of $\phi$ by including $\chi$ with the
$12$ lifts coming from the various orbits on $\BB$.
\end{proof}

Notice that this shows that there is no pair $(U,\delta)$ where
$\delta$ is $\pi$-special so that all lifts of $\phi$ can be
obtained as images of a bijection $\gamma  \mapsto (\delta\gamma)^G$
where $\gamma$ is a $\pi'$-special and linear character of $U$ since
the $\pi'$-special, linear characters will form a group isomorphic
to a subgroup of $U/U'$ and hence, the number of such characters
would divide $|U|$ and hence $|G|$.  But, we showed that $\phi$ has
$13$ lifts and $13$ does not divide $|G|$.

We now show that we can find an inductive pair that is not
self-stabilizing.  We set $\NN = \{ 1, V, G \}$, $\NN_1 = \{ 1, V,
M_1 V, G \}$, and $\NN_2 = \{ 1, V, M_2 V, G \}$. Observe that
$\chi_V = \chi (1) 1_G$, so $(G,\chi)$ is the self-stabilizing pair
for $\chi$ with respect to $\NN$. On the other hand, since for $i =
1$ or $2$ we know that $\chi$ is induced from $\delta_i \in \irr
{M_i V}$ where $\delta_i$ is an extension of $\hat\lambda$ to
$M_iV$, it follows that $(M_i V,\delta_i)$ is the self-stabilizing
pair $\chi$ with respect to $\NN_i$.  This implies that each $(M_i
V,\delta_i)$ is an example of a pair that is inductive for $\NN$ but
not a self-stabilizing pair with respect to $\NN$.  Thus, we have
found three different pairs $(G,\chi)$, $(M_1V,\delta_1)$, and
$(M_2V,\delta_2)$ inducing $\chi$ that are inductive for $\NN$.
Notice that $(M_1V,\delta_1)$ is also inductive for $\NN_1$, but is
not inductive for $\NN_2$.  Also, $(M_iV,\delta_2)$ is inductive for
$\NN_2$, but is not inductive for $\NN_1$.  Thus, the pairs $(MV_i,
\delta_i)$ are inductive for $\NN$, but not self-stabilizing, and
hence, we have our pairs that are inductive but not
self-stabilizing.

We now show that $\chi$ is the only $\pi$-lift for $\phi$ with the
property that for every normal subgroup $N$ that the irreducible
constituents of $\chi_N$ are $\pi$-lifts.  Observe that $M_1 V = V_1
\times M_1 V_2$, and so, $(M_1 V)' = V_2$.  It follows that $\card
{M_1 V/(M_1 V)'}_{3'} = 7$.  In fact, $M_1 V/ (M_1 V)' \cong M_1
\times V_1$.  Let ${\cal M}_1$ be the seven characters obtained by
$\zeta \mapsto (\beta_1 \zeta)^G$ as $\zeta$ runs through the linear
characters of $\irr {M_1V}$ with order $7$.  In other word, $\zeta$
runs through the characters in $\irr {V_1}$.  We know from Theorem
\ref{main2} that all of the characters in ${\cal M}_1$ are
$\NN_1$-$\pi$-lifts. If $1 \ne \zeta \in \irr {V_1}$, then $M_1V$ is
the stabilizer in $G$ of both $\zeta$ and $\delta_1 \zeta$.  It
follows that $ZV$ is the stabilizer in $M_2V$ of $\hat\lambda
\zeta$.  In other words, $(\hat\lambda \zeta)^{M_2V}$ is
irreducible.  Applying Mackey's theorem, the irreducible
constituents of $(\delta_1 \zeta)^G|_{M_2V}$ have degree $3$ and
thus are not $\pi$-lifts.  Thus, $\chi$ is the only character in
${\cal M}_1$ that is an $\NN_1$-$\pi$-lift.

In a similar manner, we can define ${\cal M}_2$ with respect to
$(M_2 V)$.   Working as above, we will see that ${\cal M}_2$
contains $7$ characters that are $\NN_2$-$\pi$-lifts of $\phi$, and
only, $\chi$ is also an $\NN_1$-$\pi$-lift.  It follows that ${\cal
M}_1 \cap {\cal M}_2 = \{ \chi \}$, and so, $\card {{\cal M}_1 \cup
{\cal M}_2} = 13$.  We can conclude that ${\cal M}_1 \cup {\cal
M}_2$ is the set of all lifts of $\phi$.  Notice that this implies
that all lifts of $\phi$ are $\NN$-$\pi$-lifts.  In particular, it
is possible to have $\NN$-$\pi$-lifts that do not arise from a
self-stabilizing pair.

\end{document}